\documentclass{amsart}
\usepackage{setspace}
\usepackage{a4}
\usepackage{amsthm}
\usepackage{latexsym}
\usepackage{amsfonts}
\usepackage{graphicx}
\usepackage{textcomp}
\usepackage{cite}
\usepackage{enumerate}
\usepackage{amssymb}
\usepackage{hyperref}
\usepackage{amsmath}
\usepackage{tikz}
\usepackage[mathscr]{euscript}
\usepackage{mathtools}
\newtheorem{theorem}{Theorem}[section]

\newtheorem{corollary}[theorem] {Corollary}

\newtheorem{problem}[theorem]{Problem}

\title{This is the title}
\usepackage{fancyhdr}

\begin{document}
	\vspace{0.9cm}	
\hrule\hrule\hrule\hrule\hrule
\vspace{0.3cm}	
\begin{center}
{\bf{Octonionic Kittaneh Numerical Radius Inequalities}}\\
\vspace{0.3cm}
\hrule\hrule\hrule
\vspace{0.3cm}
\textbf{K. Mahesh Krishna}\\
School of Mathematics and Natural Sciences\\
Chanakya University Global Campus\\
NH-648, Haraluru Village\\
Devanahalli Taluk, 	Bengaluru  North District\\
Karnataka State 562 110 India \\
Email: kmaheshak@gmail.com\\

Date: \today
\hrule\hrule
\end{center}

\vspace{0.5cm}
%--------------------------------------
\textbf{Abstract}: We show that for certain classes of matrices over octonions, a weak version of breakthrough numerical radius inequalities obtained by Kittaneh [\textit{{Stud. Math.}, 2005}] hold.



\hrule

%\tableofcontents
\hrule
\section{Introduction}
For $n \in \mathbb{N}$, let $\mathbb{C}^n$ be the standard Hilbert space (with the inner product linear in the first variable and conjugate linear in the second variable). Recall that 
the \textbf{numerical range} or the \textbf{field of values}  \cite{GUSTAFSONRAO, HALMOS, HORNJOHNSON, GOLDBERGTADMOR, GAUWUBOOK, TOEPLITZ, HAUSDORFF, ISTRATESCU, HALMOS2} of a matrix $T \in M_n(\mathbb{C})$ is defined as 
\begin{align*}
	W_\mathbb{C}(T)\coloneqq \{\langle Th, h\rangle: h \in \mathbb{C}^n, \|h\|=1\} \subseteq \mathbb{C}
\end{align*}
and the \textbf{numerical radius} of  $T$ is defined as 
\begin{align*}
	w_\mathbb{C}(T)\coloneqq \sup_{\lambda \in W_\mathbb{C}(T)}|\lambda|=\sup_{h \in \mathbb{C}^n, \|h\|=1} |\langle Th, h \rangle |.
\end{align*}
The Toeplitz-Hausdorff theorem says that the numerical range is always a convex subset of  $\mathbb{C}$ \cite{HAUSDORFF, TOEPLITZ, HALMOS, HORNJOHNSON, GUSTAFSONRAO, GAUWUBOOK, LI}.  Elementary calculations using the Cauchy-Schwarz inequality and the generalized polarization identity  \cite{GUSTAFSONRAO} show that
\begin{align}\label{FI}
	\frac{\|T\|}{2}\leq w_\mathbb{C}(T)\leq \|T\|.
\end{align}
 Inequalities (\ref{FI}) say that the map $w_\mathbb{C}(\cdot):M_n(\mathbb{C})\to \mathbb{R}$ is a norm which is equivalent to the operator norm.  
In 2005, Kittaneh made a far reaching generalization of Inequalities (\ref{FI}) \cite{KITTANEH}.
\begin{theorem} \cite{KITTANEH} (\textbf{Kittaneh Numerical Radius Inequalities}) \label{KT}
	Let $T \in M_n(\mathbb{C})$. Then 
\begin{align}\label{KI}
	\frac{1}{2}\sqrt{\|T^*T+TT^*\|}\leq w_\mathbb{C}(T)\leq 	\frac{1}{\sqrt{2}}\sqrt{\|T^*T+TT^*\|}.	
\end{align}
\end{theorem}
Since 
\begin{align*}
			\frac{\|T\|}{2}\leq \frac{1}{2}\sqrt{\|T^*T+TT^*\|}\leq	\frac{1}{\sqrt{2}}\sqrt{\|T^*T+TT^*\|}\leq \|T\|, \quad \forall T \in M_n(\mathbb{C}),
\end{align*}
Inequalities (\ref{KI}) improve Inequalities (\ref{FI}). In this  note, we show that for certain classes of octonionic matrices, a more restrictive version of Inequalities  (\ref{KI}) hold.  Let $\{e_j\}_{j=1}^7$ be the free symbols whose multiplication is governed by   following table \cite{BAEZ}. 
\begin{table}[h!]
	\begin{tabular}{c|ccccccc}
		$\times$ & $e_1$ & $e_2$ & $e_3$ & $e_4$ & $e_5$ & $e_6$ & $e_7$ \\ \hline
		$e_1$ & $-1$ & $e_4$ & $e_7$ & $-e_2$ & $-e_6$ & $e_5$ & $-e_3$ \\
		$e_2$ & $-e_4$ & $-1$ & $e_5$ & $e_1$ & $-e_3$ & $e_7$ & $-e_6$ \\
		$e_3$ & $-e_7$ & $-e_5$ & $-1$ & $e_6$ & $e_2$ & $-e_4$ & $e_1$ \\
		$e_4$ & $e_2$ & $-e_1$ & $-e_6$ & $-1$ & $e_7$ & $e_3$ & $-e_5$ \\
		$e_5$ & $e_6$ & $e_3$ & $-e_2$ & $-e_7$ & $-1$ & $-e_1$ & $e_4$ \\
		$e_6$ & $-e_5$ & $-e_7$ & $e_4$ & $-e_3$ & $e_1$ & $-1$ & $e_2$ \\
		$e_7$ & $e_3$ & $e_6$ & $-e_1$ & $e_5$ & $-e_4$ & $-e_2$ & $-1$\\
	\end{tabular}
\end{table}

Let  
\begin{align*}
	\mathbb{O}\coloneqq \left\{a_0+\sum_{j=1}^{7}a_je_j: a_0, \dots, a_7 \in \mathbb{R} \right\}
\end{align*}
be the  non-associative division algebra of octonions  with the conjugation
\begin{align*}
	\overline{a_0+\sum_{j=1}^{7}a_je_j}\coloneqq a_0-\sum_{j=1}^{7}a_je_j, \quad \forall a_0+\sum_{j=1}^{7}a_je_j \in 	\mathbb{O}
\end{align*}
and the modulus 
\begin{align*}
\left|a_0+\sum_{j=1}^{7}a_je_j\right|\coloneqq \sqrt{\sum_{j=0}^{7}a_j^2}, \quad \forall a_0+\sum_{j=1}^{7}a_je_j \in 	\mathbb{O}.	
\end{align*}

Let $\mathbb{O}^n$ be the standard right octonionic  Hilbert space equipped with the inner product 
\begin{align*}
	\langle (z_r)_{r=1}^n, (w_r)_{r=1}^n\rangle \coloneqq \sum_{r=1}^{n}\overline{z_r}w_r, \quad \forall (z_r)_{r=1}^n, (w_r)_{r=1}^n\in \mathbb{O}^n
\end{align*}
and the norm 
\begin{align*}
	\|(z_r)_{r=1}^n\|\coloneqq \left(\sum_{r=1}^{n}|z_r|^2\right)^\frac{1}{2}, \quad \forall (z_r)_{r=1}^n\in \mathbb{O}^n.
\end{align*}
Given a matrix $T \in M_n(\mathbb{O})$, its operator norm is defined as 
\begin{align*}
	\|T\|\coloneqq \sup_{h \in \mathbb{O}^n, \|h\|=1}\|Th\|.
\end{align*}
We derive octonionic version of Theorem \ref{KT} in Theorem \ref{QKT}.

\section{Octonionic Kittaneh Numerical Radius Inequalities}

The \textbf{octonionic numerical range} of an octonionic matrix $T \in M_n(\mathbb{O})$ \cite{RODMAN} is defined as 
\begin{align*}
	W_\mathbb{O}(T)\coloneqq \{\langle h,Th\rangle: h \in \mathbb{O}^n, \|h\|=1\}\subseteq \mathbb{O}
\end{align*}
and the \textbf{octonionic numerical radius} of $T$ is defined as 
\begin{align*}
	w_\mathbb{O}(T)\coloneqq \sup_{\lambda \in W_\mathbb{O}(T)}|\lambda|=\sup_{h \in \mathbb{O}^n, \|h\|=1}|\langle h, Th\rangle|.
\end{align*}
We first derive octonionic version of Inequalities (\ref{FI}). 
\begin{theorem}
Let $T \in M_n(\mathbb{O})$. Then 	
\begin{align*}
		\frac{\|T\|}{8}\leq w_\mathbb{O}(T)\leq \|T\|.
\end{align*}
\end{theorem}
\begin{proof}
	Clearly second inequality comes from Cauchy-Schwarz inequality. Let $x, y \in \mathbb{O}^n.$ By the generalized polarization identity, 
	\begin{align*}
		4 \langle y, Tx \rangle =& \langle (x+y), T(x+y) \rangle- \langle (x-y), T(x-y) \rangle\\
		&+\sum_{j=1}^{7}(\langle (x+ye_j), T(x+ye_j) \rangle- \langle (x-ye_j), T(x-ye_j))e_j.
	\end{align*}
Hence
\begin{align*}
	4| \langle y, Tx \rangle|&\leq w_\mathbb{O}(T)\left(\|x+y\|^2+\|x-y\|^2+\sum_{j=1}^{7}(\|x+ye_j\|^2+\|x-ye_j\|^2)\right)\\
	&=w_\mathbb{O}(T)\left(2\|x\|^2+2\|y\|^2+\sum_{j=1}^{7}(2\|x\|^2+2\|y\|^2)\right)\\
	&=16w_\mathbb{O}(T)(\|x\|^2+\|y\|^2).
\end{align*}
 Therefore 
 \begin{align*}
 	\|T\|=\sup_{x, y \in \mathbb{O}^n, \|x\|=\|y\|=1}|\langle y, Tx\rangle|\leq 8 w_\mathbb{O}(T).
 \end{align*}
\end{proof}
Let $T \in M_n(\mathbb{O})$ be self-adjoint.  We note that
\begin{align}\label{ST}
	\|T\|=\sup_{h \in \mathbb{O}^n, \|h\|=1}|\langle h, Th\rangle|.
\end{align}
In fact, define $M\coloneqq \sup_{h \in \mathbb{O}^n, \|h\|=1}|\langle h, Th\rangle|$. The we have
\begin{align*}
	M= \sup_{h \in \mathbb{O}^n, \|h\|=1}|\langle h, Th\rangle|\leq \sup_{h \in \mathbb{O}^n, \|h\|=1}\|h\| \|Th\|=\sup_{h \in \mathbb{O}^n, \|h\|=1}\|Th\|=\|T\|.
\end{align*}
Now let $x, y \in \mathbb{O}^n$. Then 
\begin{align*}
	\langle x+y, T(x+y)\rangle -\langle x-y, T(x-y)\rangle =2(\langle x, Ty \rangle +\langle y, Tx \rangle).
\end{align*}
Therefore
\begin{align*}
	|\langle x, Ty \rangle +\langle y, Tx \rangle|&\leq \frac{1}{2}(|\langle x+y, T(x+y)\rangle|+|\langle x-y, T(x-y)\rangle|)\\
	&\leq  \frac{1}{2}(M\|x+y\|^2+M\|x-y\|^2)\\
	&=M(\|x\|^2+\|y\|^2).
\end{align*}
Since, $T$ is self-adjoint, 
\begin{align}\label{AE}
		|\langle x, Ty \rangle +\langle Ty, x \rangle|\leq M(\|x\|^2+\|y\|^2), \quad \forall x, y \in \mathbb{O}^n.
\end{align}
For a given $h \in \mathbb{O}^n$ with $ \|h\|=1$ and $\|Th\|\neq 0$,  we take 
\begin{align*}
	x=\frac{Th}{\|Th\|}, \quad y = h
\end{align*}
in Inequality (\ref{AE}). Then we have 
\begin{align*}
	2\|Th\|\leq 2M.
\end{align*}
By varying $h\in \mathbb{O}^n$ with $ \|h\|=1$ we get $\|T\|\leq M$. Therefore $\|T\|=M$. \\
We now derive our main theorem using Equation (\ref{ST}).
 In the paper, the anti-commutator of $T, S \in M_n(\mathbb{O})$ is defined as $\{T, S\}\coloneqq TS+ST$.
\begin{theorem}\label{QKT}
(\textbf{Octonionic Kittaneh Numerical Radius Inequalities})
Let $T \in M_n(\mathbb{O})$. Assume that 
\begin{align*}
	T=A_0+\sum_{j=1}^{7}e_jA_j \quad \text{ for some } A_0, \dots, A_7 \in M_n(\mathbb{O})
\end{align*}
satisfying the following conditions. 
	\begin{enumerate}[\upshape(i)]
		\item $A^*_j=A_j$ for all $0\leq j \leq 7$.
	\item $e_jA_j=A_je_j$ for all $1\leq j \leq 7$.
	\item $\sum_{j,k=1, j \neq k}^{7}\{e_jA_j, e_kA_k\}=0$.
		\end{enumerate}
Then 
\begin{align}\label{MI}
	\frac{\|T\|}{8}\leq	\frac{1}{4\sqrt{2}}\sqrt{\|T^*T+TT^*\|}\leq w_\mathbb{O}(T)\leq\frac{1}{\sqrt{2}}\sqrt{\|T^*T+TT^*\|}\leq \|T\|.	
\end{align}	
\end{theorem}
\begin{proof}
	We have 
	\begin{align*}
		T=A_0+\sum_{j=1}^{7}e_jA_j, \quad T^*=A_0-\sum_{j=1}^{7}e_jA_j.
	\end{align*}
Therefore 
\begin{align*}
	T^*T+TT^*&=\left(A_0-\sum_{j=1}^{7}e_jA_j\right)\left(A_0+\sum_{j=1}^{7}e_jA_j\right)+\left(A_0+\sum_{j=1}^{7}e_jA_j\right)\left(A_0-\sum_{j=1}^{7}e_jA_j\right)\\
	&=2\left(\sum_{j=0}^{7}A^2_j\right)-2\sum_{j,k=1, j \neq k}^{7}\{e_jA_j, e_kA_k\}\\
	&=2\left(\sum_{j=0}^{7}A^2_j\right).
\end{align*}
For $h \in \mathbb{O}^n$ with $\|h\|=1$, we have
\begin{align*}
|\langle h, Th \rangle |^2&=\left|\left\langle h, \left(A_0+\sum_{j=1}^{7}e_jA_j\right)h \right\rangle \right|^2\\	
&=\left|\left\langle h, A_0h \right \rangle+ \sum_{j=1}^{7}e_j\left\langle h, A_jh \right \rangle\right|^2\\
&=\sum_{j=0}^{7}|\langle h, A_jh  \rangle|^2\\
&\leq \sum_{j=0}^{7}\|A_jh\|^2\\
&= \sum_{j=0}^{7}\langle h, A_j^2h \rangle \\
&=\left\langle h, \left(\sum_{j=0}^{7}A_j^2\right)h \right \rangle
\end{align*}
Since all $A_j$'s  are self-adjoint, we have
\begin{align*}
	w_\mathbb{O}(T)^2&= \sup_{h \in \mathbb{O}^n, \|h\|=1} |\langle h, Th \rangle |^2\\
	&\leq \sup_{h \in \mathbb{O}^n, \|h\|=1} \left\langle h, \left(\sum_{j=0}^{7}A_j^2\right)h \right\rangle\\
	&=\left\|\sum_{j=0}^{7}A_j^2\right\|\\
	&=\frac{1}{2}\|	T^*T+TT^*\|.
\end{align*}
On the other hand, for $h \in \mathbb{O}^n$ with $\|h\|=1$, using the Cauchy-Schwarz inequality, we also have 
\begin{align*}
|\langle h, Th \rangle |^2&=\left|\left\langle h, \left(A_0+\sum_{j=1}^{7}e_jA_j\right)h \right\rangle \right|^2\\	
&=\left|\left\langle h, A_0h \right \rangle+ \sum_{j=1}^{7}e_j\left\langle h, A_jh \right \rangle\right|^2\\
&=\sum_{j=0}^{7}|\langle h, A_jh  \rangle|^2\\
&\geq 	\frac{1}{8}\left(\sum_{j=0}^{7}|\langle h, A_jh  \rangle|\right)^2\\
&\geq 	\frac{1}{8}\left|\sum_{j=0}^{7}\langle h, (\pm A_j)h  \rangle\right|^2\\
&=\frac{1}{8}\left|\left\langle h, \sum_{j=0}^{7}(\pm A_jh) \right\rangle\right |^2.
\end{align*}
Since all $A_j$'s are self-adjoint, we have
\begin{align*}
	w_\mathbb{O}(T)^2&= \sup_{h \in \mathbb{O}^n, \|h\|=1} |\langle h, Th \rangle |^2\\
	&\geq \frac{1}{8}\sup_{h \in \mathbb{O}^n, \|h\|=1}\left|\left\langle h, \sum_{j=0}^{7}(\pm A_jh) \right\rangle\right  |^2\\
	&= \frac{1}{8}\left\|\sum_{j=0}^{7}(\pm A_j)\right\|^2\\
	&=\frac{1}{8}\left\|\left(\sum_{j=0}^{7}(\pm A_j)\right)^2\right\|.
\end{align*}
Define 
\begin{align*}
	S\coloneqq \sum_{j=0}^{7}A_j.
\end{align*}
Then 
\begin{align*}
		64w_\mathbb{O}(T)^2&\geq \sum_{k=0}^{7}\|(S-2A_k)^2\|\\
		&\geq \left\| \sum_{k=0}^{7}(S-2A_k)^2\right\|\\
		&=\left\|\sum_{k=0}^{7}(S^2+4A_k^2-2SA_k-2A_kS)\right\|\\
		&=\left\|\sum_{k=0}^{7}S^2+4\sum_{k=0}^{7}A_k^2-2S\sum_{k=0}^{7}A_k-2\sum_{k=0}^{7}A_kS\right\|\\
		&=\left\|8S^2+4\sum_{k=0}^{7}A_k^2-2S^2-2S^2\right\|\\
		&=\left\|4S^2+4\sum_{k=0}^{7}A_k^2\right\|\\
		&\geq 4\left\|\sum_{k=0}^{7}A_k^2\right\|\\
		&=2\|T^*T+TT^*\|.
\end{align*}
\end{proof}
\begin{corollary}
Let $T \in M_n(\mathbb{O})$. Assume that 
\begin{align*}
T=A_0+\sum_{j=1}^{7}e_jA_j \quad \text{ for some } A_0, \dots, A_7 \in M_n(\mathbb{O})
\end{align*}
satisfying the following conditions. 
	\begin{enumerate}[\upshape(i)]
	\item $A^*_j=A_j$ for all $0\leq j \leq 7$.
	\item $e_jA_j=A_je_j$ for all $1\leq j \leq 7$.
	\item $\{e_jA_j, e_kA_k\}=0$ for all $1\leq j \neq k \leq 7$.
\end{enumerate}
Then 
\begin{align*}
	\frac{1}{4\sqrt{2}}\sqrt{\|T^*T+TT^*\|}\leq w_\mathbb{O}(T)\leq\frac{1}{\sqrt{2}}\sqrt{\|T^*T+TT^*\|}.	
\end{align*}		
\end{corollary}
Following corollary says that using self adjoint real matrices, we can easily construct octonionic matrices which satisfy  Inequalities (\ref{MI}). 
\begin{corollary}
Let $A_0, \dots, A_7 \in M_n(\mathbb{R})$ satisfying the  following conditions.  
\begin{enumerate}[\upshape(i)]
	\item $A_j^t=A_j$ for all $0\leq j \leq 7$.
	\item $\sum_{j,k=1, j \neq k}^{7}\{A_j, A_k\}=0$.
\end{enumerate}
Define 
\begin{align*}
	T\coloneqq A_0+\sum_{j=1}^{7}e_jA_j \in M_n(\mathbb{O}).
\end{align*}
Then 
\begin{align*}
	\frac{1}{4\sqrt{2}}\sqrt{\|T^*T+TT^*\|}\leq w_\mathbb{O}(T)\leq\frac{1}{\sqrt{2}}\sqrt{\|T^*T+TT^*\|}.	
\end{align*}		
\end{corollary}
Kittaneh actually  showed that Theorem \ref{KT} holds for any bounded linear operator on a complex Hilbert space. We state here the octonionic Hilbert space version of Theorem \ref{QKT}, whose proof is similar to that of Theorem \ref{QKT}. For basics on octonionic Hilbert spaces, we refer \cite{GOLDSTINEHORWITZ}.
\begin{theorem}
	Let $\mathcal{Q}$ be a octonionic Hilbert space. Let $T:\mathcal{Q} \to \mathcal{Q} $ be a bounded linear operator. Assume that 
	\begin{align*}
		T=A_0+\sum_{j=1}^{7}e_jA_j, \text{ for some bounded linear operators } A_0, \dots, A_7 : \mathcal{Q} \to \mathcal{Q}
	\end{align*}
	satisfying the following conditions. 
	\begin{enumerate}[\upshape(i)]
		\item $A^*_j=A_j$ for all $0\leq j \leq 7$.
		\item $e_jA_j=A_je_j$ for all $1\leq j \leq 7$.
	\item $\sum_{j,k=1, j \neq k}^{7}\{e_jA_j, e_kA_k\}=0$.
	\end{enumerate}
	Then 
	\begin{align*}
		\frac{1}{4\sqrt{2}}\sqrt{\|T^*T+TT^*\|}\leq w_\mathbb{O}(T)\leq\frac{1}{\sqrt{2}}\sqrt{\|T^*T+TT^*\|}.	
	\end{align*}
\end{theorem}
\begin{corollary}
	Let $\mathcal{H}$ be a real Hilbert space. Let  $A_0, \dots, A_7 :\mathcal{H} \to \mathcal{H} $ be  bounded linear operators satisfying the following conditions. 
	\begin{enumerate}[\upshape(i)]
		 	\item $A^*_j=A_j$ for all $0\leq j \leq 7$.
		\item $\sum_{j,k=1, j \neq k}^{7}\{A_j, A_k\}=0$.
	\end{enumerate}
Define 
\begin{align*}
	T\coloneqq A_0+\sum_{j=1}^{7}e_jA_j.
\end{align*}
	Then 
	\begin{align*}
		\frac{1}{4\sqrt{2}}\sqrt{\|T^*T+TT^*\|}\leq w_\mathbb{O}(T)\leq\frac{1}{\sqrt{2}}\sqrt{\|T^*T+TT^*\|}.	
	\end{align*}	
\end{corollary}
We remark that Abu-Omar and Kittaneh \cite{ABUOMARKITTANEH} and Bhunia, Bag and Paul \cite{BHUNIABAGPAUL} improved Theorem \ref{KT} but the ideas used by them do not seem to carry over to octonionic Hilbert spaces.

\section{Conclusions}
	\begin{enumerate}[\upshape(i)]
	\item In 2005, Kittaneh derived a breakthrough improvement of classical numerical radius inequalities \cite{KITTANEH}.
	\item In this article, we showed that a constrained version of Kittaneh numerical radius inequalities can be obtained for certain classes of matrices over octonions.
\end{enumerate}

\section{Octonionic Crouzeix Problem}
Surprising Crouzeix result says the following. 
\begin{theorem} \label{CT} \cite{CROUZEIX, CROUZEIX2, CROUZEIX3, BICKELGORKIN, RANSFORDSCHWENNINGER} (\textbf{Crouzeix Theorem}) 
	For every $d \in \mathbb{N}$ and for every matrix $M\in \mathbb{M}_{d}(\mathbb{C})$, we have 
	\begin{align*}
		(\text{\textbf{Crouzeix-Palencia Inequality}})	\quad \quad \quad 	\|p(M)\|\leq (1+\sqrt{2}) \sup\left\{|p(z)|: z \in W_\mathbb{C}(M)\right\},  \quad \forall p \in \mathbb{C}[z].
	\end{align*}	
\end{theorem}

Based on Theorem \ref{CT}, we formulate the following problem. 
\begin{problem} (\textbf{Octonionic Crouzeix Problem})
	Whether there is a universal constant $R_\mathbb{O}$ satisfying the following: For every $d \in \mathbb{N}$ and for every $A\in \mathbb{M}_{d}(\mathbb{O})$, we have 
	\begin{align*}
		\|p(A)\|\leq R_\mathbb{O} \sup\left\{|p(z)|: z \in W_\mathbb{O}(A)\right\},  \quad \forall p(z)\coloneqq a_0+a_1z+a_2z^2+\cdots +a_nz^n \in \mathbb{O}[z], n \in \mathbb{N}
	\end{align*}
	where $\mathbb{O}[z]$ is the set of all polynomials over $\mathbb{O}$. 
	We conjecture that $R_\mathbb{O}$ exists and $R_\mathbb{O}=8$.
\end{problem}

\section{Acknowledgments}
AI is used for the literature survey and grammar improvement.

 \bibliographystyle{plain}
 \bibliography{reference.bib}

\begin{thebibliography}{10}

\bibitem{ABUOMARKITTANEH}
Amer Abu-Omar and Fuad Kittaneh.
\newblock Upper and lower bounds for the numerical radius with an application
  to involution operators.
\newblock {\em Rocky Mt. J. Math.}, 45(4):1055--1064, 2015.

\bibitem{BAEZ}
John~C. Baez.
\newblock The octonions.
\newblock {\em Bull. Am. Math. Soc., New Ser.}, 39(2):145--205, 2002.

\bibitem{BHUNIABAGPAUL}
Pintu Bhunia, Santanu Bag, and Kallol Paul.
\newblock Numerical radius inequalities and its applications in estimation of
  zeros of polynomials.
\newblock {\em Linear Algebra Appl.}, 573:166--177, 2019.

\bibitem{BICKELGORKIN}
Kelly Bickel, Pamela Gorkin, Anne Greenbaum, Thomas Ransford, Felix~L.
  Schwenninger, and Elias Wegert.
\newblock Crouzeix's conjecture and related problems.
\newblock {\em Comput. Methods Funct. Theory}, 20(3-4):701--728, 2020.

\bibitem{CROUZEIX3}
M.~Crouzeix and C.~Palencia.
\newblock The numerical range is a {{\((1+\sqrt{2})\)}}-spectral set.
\newblock {\em SIAM J. Matrix Anal. Appl.}, 38(2):649--655, 2017.

\bibitem{CROUZEIX}
Michel Crouzeix.
\newblock Bounds for analytical functions of matrices.
\newblock {\em Integral Equations Oper. Theory}, 48(4):461--477, 2004.

\bibitem{CROUZEIX2}
Michel Crouzeix.
\newblock Numerical range and functional calculus in {Hilbert} space.
\newblock {\em J. Funct. Anal.}, 244(2):668--690, 2007.

\bibitem{GAUWUBOOK}
Hwa-Long Gau and Pei~Yuan Wu.
\newblock {\em Numerical ranges of {Hilbert} space operators}, volume 179 of
  {\em Encycl. Math. Appl.}
\newblock Cambridge: Cambridge University Press, 2021.

\bibitem{GOLDBERGTADMOR}
Moshe Goldberg and Eitan Tadmor.
\newblock On the numerical radius and its applications.
\newblock {\em Linear Algebra Appl.}, 42:263--284, 1982.

\bibitem{GOLDSTINEHORWITZ}
H.~H. Goldstine and L.~P. Horwitz.
\newblock Hilbert space with non-associative scalars. {I}.
\newblock {\em Math. Ann.}, 154:1--27, 1964.

\bibitem{GUSTAFSONRAO}
Karl~E. Gustafson and Duggirala K.~M. Rao.
\newblock {\em Numerical range. {The} field of values of linear operators and
  matrices}.
\newblock Universitext. New York, NY: Springer, 1996.

\bibitem{HALMOS2}
P.~R. Halmos.
\newblock Numerical ranges and normal dilations.
\newblock {\em Acta Sci. Math.}, 25:1--5, 1964.

\bibitem{HALMOS}
Paul~R. Halmos.
\newblock {\em A {Hilbert} space problem book}, volume~19 of {\em Grad. Texts
  Math.}
\newblock Springer, Cham, 1982.

\bibitem{HAUSDORFF}
F.~Hausdorff.
\newblock Der {Wertvorrat} einer {Bilinearform}.
\newblock {\em Math. Z.}, 3:314--316, 1919.

\bibitem{HORNJOHNSON}
Roger~A. Horn and Charles~R. Johnson.
\newblock {\em Topics in matrix analysis}.
\newblock Cambridge University Press, 1991.

\bibitem{ISTRATESCU}
Vasile~I. Istratescu.
\newblock {\em Introduction to linear operator theory}, volume~65 of {\em Pure
  Appl. Math., Marcel Dekker}.
\newblock Marcel Dekker, Inc., New York, NY, 1981.

\bibitem{KITTANEH}
Fuad Kittaneh.
\newblock Numerical radius inequalities for {Hilbert} space operators.
\newblock {\em Stud. Math.}, 168(1):73--80, 2005.

\bibitem{LI}
Chi-Kwong Li.
\newblock A simple proof of the elliptical range theorem.
\newblock {\em Proc. Am. Math. Soc.}, 124(7):1985--1986, 1996.

\bibitem{RANSFORDSCHWENNINGER}
Thomas Ransford and Felix~L. Schwenninger.
\newblock Remarks on the {Crouzeix}-{Palencia} proof that the numerical range
  is a $(1+\sqrt{2})$-spectral set.
\newblock {\em SIAM J. Matrix Anal. Appl.}, 39(1):342--345, 2018.

\bibitem{RODMAN}
Leiba Rodman.
\newblock Hermitian octonion matrices and numerical ranges.
\newblock {\em Electron. J. Linear Algebra}, 27:515--533, 2014.

\bibitem{TOEPLITZ}
O.~Toeplitz.
\newblock Das algebraische {Analogon} zu einem {Satze} von {Fej{\'e}r}.
\newblock {\em Math. Z.}, 2:187--197, 1918.

\end{thebibliography}

\end{document}